\documentclass[11pt,reqno,a4paper]{amsart}

\usepackage{amsmath}
\usepackage{enumerate}
\usepackage{amssymb}
\usepackage{amscd}
\usepackage{amsthm}
\usepackage{amsfonts}
\usepackage{graphicx}
\usepackage{hyperref}

\usepackage{fouriernc}
\usepackage[T1]{fontenc}

\numberwithin{equation}{section}

\newcommand{\cF}{\mathcal{F}}

\newcommand{\cH}{\mathcal{H}}

\newcommand{\bR}{\mathbb{R}}

\newcommand{\ra}{\rightarrow}

\newcommand{\qand}{\quad \textrm{and} \quad}

\def\acts{\curvearrowright}
\newcommand\subsetsim{\mathrel{%
\ooalign{\raise0.2ex\hbox{$\subset$}\cr\hidewidth\raise-0.8ex\hbox{\scalebox{0.9}{$\sim$}}\hidewidth\cr}}}

\DeclareMathOperator{\supp}{supp}

\DeclareMathOperator{\linspan}{span}

\theoremstyle{theorem}
\newtheorem{theorem}{Theorem}[section]
\newtheorem{corollary}[theorem]{Corollary}

\newtheorem{lemma}[theorem]{Lemma}

\theoremstyle{definition}

\newtheorem{remark}{Remark}[section]

\begin{document}

\title{Random Walks on Countable Groups}

\author{Michael Bj\"orklund}
\address{Department of Mathematics, Chalmers, Gothenburg, Sweden}
\email{micbjo@chalmers.se}
\thanks{}



\date{}

\dedicatory{}

\begin{abstract}
We begin by giving
a new proof of the equivalence between the Liouville property and vanishing of the drift for symmetric 
random walks with finite first moments on finitely generated groups; a result which was first established by Kaimanovich-Vershik and Karlsson-Ledrappier. 
We then proceed to prove that the product of the Poisson boundary of any countable measured group $(G,\mu)$
with any ergodic $(G,\check{\mu})$-space is still ergodic, which in particular yields a new proof of weak mixing 
for the double Poisson boundary of $(G,\mu)$ when $\mu$ is symmetric. Finally, we characterize the failure of weak-mixing for an ergodic $(G,\mu)$-space as the existence of a non-trivial measure-preserving isometric factor.
\end{abstract}

\maketitle

\section{Measured groups and their Poisson boundaries}

A \emph{measured group} is a pair $(G,\mu)$, where $G$ is a countable group and $\mu$ is a probability measure
on $G$ whose support generates $G$ as a semi-group. We say that $\mu$ is \emph{symmetric} if the 
\emph{adjoint} probability
measure $\check{\mu}(s) := \mu(s^{-1})$ coincides with $\mu$. We write $\mu^{0} = \delta_e$ and 
\[
\mu^{*n}(s) = \sum \mu(s_1) \cdots \mu(s_n), \quad \textrm{for $s \in G$ and $n \geq 1$},
\]
where the sum is taken over all $n$-tuples $(s_1,\ldots,s_n)$ in $G^n$ such that $s_1 \cdots s_n = s$. A real-valued function $u$ on $G$ is \emph{$\mu$-harmonic} if 
\[
\sum_{s} u(gs) \mu(s) = u(g), \quad \textrm{for all $g \in G$}.
\]
Let $\ell^\infty(G)$ denote the Banach space of real-valued \emph{bounded} functions on $G$ endowed with the 
supremum norm, and let $H^\infty(G,\mu)$ denote the closed sub-space of $\ell^\infty(G)$ consisting of 
bounded $\mu$-harmonic functions. We say that $(G,\mu)$ is \emph{Liouville} if $H^\infty(G,\mu)$ consists 
only of constant functions. 

Let $(X,\nu)$ be a Borel probability measure space, and denote by $L^\infty(X,\nu)$ the Banach space of 
$\nu$-essentially bounded real-valued functions on $X$, identified up to $\nu$-null sets, endowed with the (essential) 
supremum norm. Suppose that $X$ is equipped
with an action of $G$ by measurable maps, which preserve the class of 
$\nu$-null sets in $X$, and whose inverses are also measurable.  We say the action
is \emph{ergodic} if whenever $B \subset X$ is a Borel set such that $\nu(B \Delta g \cdot B) = 0$ for all $g \in G$, 
where $\Delta$ denotes the symmetric difference of sets, then $B$ is either a $\nu$-null set or a $\nu$-conull set. 
We say that $\nu$ is \emph{$\mu$-stationary}, and that $(X,\nu)$ is a \emph{$(G,\mu)$-space}, if 
\[
\sum_{s \in G} \Big( \int_X f(s x) \, d\nu(x) \Big) \mu(s) = \int_X f(x) \, d\nu(x), 
\quad \textrm{for all $f \in L^\infty(X,\nu)$}.
\]

We note that if $(X,\nu)$ is a $(G,u)$-space and $f \in L^\infty(X,\nu)$, then the image of the bounded 
linear map $P_\nu : L^\infty(X,\nu) \ra \ell^\infty(G)$ defined by
\[
P_\nu f(s) := \int_X f(s x) \, d\nu(x), \quad \textrm{for $s \in G$ and $f \in L^\infty(X,\nu)$},
\]
is contained in $H^\infty(G,\mu)$. 

\begin{theorem}[Furstenberg \cite{Fu71}]
For every measured group $(G,\mu)$ there exists an ergodic $(G,\mu)$-space $(Z,m)$ such that the bounded
linear map $P_m : L^\infty(Z,m) \ra H^\infty(G,\mu)$ defined above is an isometric isomorphism of Banach spaces. 

In particular, $(G,\mu)$ is Liouville if and only if $(Z,m)$ is trivial, i.e. if the support of $m$ consists of one point, which happens if and only if $m$ is $G$-invariant.
\end{theorem}

\begin{remark}
Furthermore, up to $G$-equivariant measurable isomorphisms, $(Z,m)$ is uniquely determined, and we 
shall refer to any representative of $(Z,m)$ as the \emph{Poisson boundary} of $(G,\mu)$.
\end{remark}

Let $L^1(Z,m)$ denote the Banach space of $\nu$-integrable functions on $Z$, identified up to null sets,
endowed with the $L^1$-norm. Using Hahn-Banach's Theorem and the fact that $L^1(Z,m)^* \cong L^\infty(Z,m)$, 
we get the following reformulation of Furstenberg's Theorem:
\begin{equation}
\label{corfur}
\overline{\linspan} \big\{ \frac{ds_*m}{dm} \, : \, s \in G \big\} = L^1(Z,m).
\end{equation}

\section{Zero drift vs. Liouville}
Suppose that $\rho$ is a \emph{semi-norm} on $G$, i.e. $\rho$ is a non-negative function on $G$ such that 
\begin{equation}
\label{norm}
\rho(e) = 0 \qand \rho(s) = \rho(s^{-1}) \qand \rho(st) \leq \rho(s) + \rho(t), \quad \textrm{for all $s, t \in G$}.
\end{equation}
If, in addition, $\rho(s) = 0$ implies that $s = e$, then we say that $\rho$ is a \emph{norm} on $G$. 
For instance, if $G$ is finitely generated and $S \subset G$ is a finite generating set with $S^{-1} = S$, then 
\[
\rho_S(s) = \inf\{ n \geq 1 \, : \, s \in S^n \big\}, \quad \textrm{for $s \in G$},
\] 
is a norm, often referred to as the \emph{word-norm} associated to $S$. Given a semi-norm $\rho$ on $G$, we define the \emph{drift} $\ell_\rho(\mu)$ of the triple $(G,\mu,\rho)$ by
\[
\ell_\rho(\mu) = \lim_n \frac{1}{n} \sum_{s} \rho(s) \, \mu^{*n}(s).
\]
The limit exists by sub-additivity and is finite if $\rho$ is $\mu$-integrable. The aim of this section is to give
a new proof of the following theorem:
\begin{theorem}
\label{thm1}
Let $(G,\mu)$ be a finitely generated symmetric measured group and let $\rho$ be a word-norm. Suppose that $\rho$ is $\mu$-integrable. Then $(G,\mu)$ is Liouville if and only if $\ell_{\rho}(\mu) = 0$.
\end{theorem}

\begin{remark}
The direction "Zero drift implies Liouville" was proved by Kaimanovich-Vershik in \cite{KV} using the Avez entropy
of random walks, and the direction "Liouville implies Zero drift" was proved by Karlsson-Ledrappier in \cite{KL} 
using their Multiplicative Ergodic Theorem. An alternative proof was later given by Erschler-Karlsson in \cite{EK}.
\end{remark}

\subsection{The Poisson semi-norm}

Let $(Z,m)$ be the Poisson boundary of $(G,\mu)$ and define the (multiplicative) \emph{Poisson cocycle} 
$\sigma(s,z) = \frac{ds_*m}{dm}(z)$, which is well-defined on a $G$-invariant $m$-conull subset $Z' \subset Z$.  
One readily verifies the relations
\begin{equation}
\label{cocycle}
\sigma(st,z) = \sigma(s,z) \sigma(t,s^{-1}z), \quad \textrm{for all $s,t \in G$ and $m$-a.e. $z$},
\end{equation}
and 
\[
\sum_{s \in G} \sigma(s,z) \, \mu^{*k}(s) = 1, \quad \textrm{for all $k \geq 1$ and $m$-a.e. $z$}.
\]
These relations in particular imply that 
\[
\rho_\mu(s) = \log \|\sigma(s,\cdot)\|_\infty, \quad \textrm{for $s \in G$},
\]
defines a semi-norm on $G$, which we shall refer to as the \emph{Poisson semi-norm} of $(G,\mu)$. 

\subsection{Zero drift implies Liouville}
Let $(G,\mu)$ be a finitely generated measured group and let $\rho$ be a word-norm on $G$. Suppose 
that $\rho$ is $\mu$-integrable and satisfies $\ell_\rho(\mu) = 0$. There exists a constant $C_\mu$ such that 
$\rho_\mu \leq C_\mu \rho$, where  $\rho_\mu$ is the Poisson semi-norm defined above, so in 
particular, we conclude that $\rho_\mu$ is $\mu$-integrable and $\ell_{\rho_\mu}(\mu) = 0$. 

Let $(Z,m)$ denote the Poisson boundary of $(G,\mu)$ and define the sequence
\[
c_n := \sum_{s} \Big( \int_Z \log \sigma(s,z) \, dm(z) \Big) \, d\check{\mu}^{*n}(s), \quad \textrm{for $n \geq 1$},
\]
where $\sigma$ denotes the Poisson cocycle defined above. One readily verifies that $c_{n+m} = c_{n} + c_{m}$ 
for all $m,n \geq 1$, and thus
\begin{eqnarray*}
0 &=& \ell_{\rho_\mu}(\mu) = \ell_{\rho_\mu}(\check{\mu}) = \lim_n \frac{1}{n} \sum_{s} \log \|\sigma(s,\cdot)\|_\infty \, \check{\mu}^{*n}(s) 
\geq \liminf_n - \frac{1}{n} \sum_s \Big(\int_Z \log \sigma(s,z) \, dm(z)\Big) \check{\mu}^{*n}(s) \\
&=&  \liminf_n - \frac{c_n}{n} = -c_1 = - \sum_s \Big(\int_Z \log \sigma(s,z) \, dm(z)\Big) \check{\mu}(s) 
\geq - \sum_s \log \Big( \int_Z \sigma(s,z) \, dm(z) \Big) \check{\mu}(s) = 0,
\end{eqnarray*}
where we in the second to last step used Jensen's inequality. We conclude that 
\[
\sum_s \Big(\int_Z \log \sigma(s,z) \, dm(z)\Big) \check{\mu}(s) = 0,
\]
and thus $\sigma(s,z) = 1$ $m$-almost everywhere, for all $s \in \supp \check{\mu}$. By \eqref{cocycle} we conclude that 
$\sigma(s,\cdot) = 1$ $m$-almost everywhere for all $s \in G$, and thus $m$ is $G$-invariant. \\

We stress that we did not use the assumption that $\mu$ is symmetric in this proof. In particular, 
since $\ell_\mu(\rho) = \ell_\rho(\check{\mu})$, we have the following corollary.

\begin{corollary}
Let $(G,\mu)$ be a finitely generated measured group and let $\rho$ be a word-norm on $G$. Assume 
that $\rho$ is $\mu$-integrable. Then $(G,\mu)$ is Liouville if and only if $(G,\check{\mu})$ is Liouville. 
\end{corollary}

There are examples (see e.g. \cite{KV}) of finitely generated measured groups which are 
Liouville, while their adjoint measured groups are not. 

\subsection{Liouville implies zero drift}
A real-valued function $\phi$ on $G$ is called \emph{left Lipschitz} if
\[
\sup_s \big| \phi(gs) - \phi(s) \big| < \infty, \quad \textrm{for all $g \in G$},
\]
and \emph{quasi-$\mu$-harmonic with distortion $\ell$}, where $\ell$ is a real number, if $\phi$ is $\mu$-integrable and
satisfy
\[
\sum_s \phi(gs) \, \mu(s) = \phi(g) + \ell, \quad \textrm{for all $g \in G$}.
\]
In particular, if $\ell = 0$, then $\phi$ is $\mu$-harmonic. The remaining direction in Theorem \ref{thm1} is now 
readily deduced  from the following two lemmata.
\begin{lemma}
\label{2.1}
Let $\rho$ be a norm on $G$ which is $\mu$-integrable. Then there exists a left Lipschitz quasi-$\mu$-harmonic function on $G$ with $\phi(e) = 0$ and distortion $\ell_\rho(\mu)$.
\end{lemma}

\begin{lemma}
\label{2.2}
Suppose that $(G,\mu)$ is Liouville. If $\phi$ is a left Lipschitz quasi-$\mu$-harmonic function on $G$ 
with distortion $\ell$ and $\phi(e) = 0$, then $\phi$ is a homomorphism from $G$ into $\bR$. In particular, 
if $\mu$ is symmetric, then $\ell = 0$.
\end{lemma}

\begin{proof}[Proof of Lemma \ref{2.1}]
By telescoping, we have
\begin{equation}
\label{alt}
\ell_\rho(\mu) 
= 
\lim_n \frac{1}{n} \sum_{k=0}^{n-1}  \sum_{s,t} \big( \rho(st) - \rho(t) \big) \, \mu^{*k}(t) \, \mu(s).
\end{equation}
Define the functions
\[
f_k(s) := \sum_t \big( \rho(st) - \rho(t) \big) \, \mu^{*k}(t), \quad \textrm{for $s \in G$},
\]
and note that by the triangle inequality, $|f_k(s)| \leq \rho(s)$ for all $s$. Since $\rho$ is assumed to be 
$\mu$-integrable, the sequence
\[
\phi_n = \frac{1}{n} \sum_{k=0}^{n-1} f_k(s), \quad \textrm{for $s \in G$},
\]
is $\mu$-dominated, and by a simple diagonal argument, we can find a sub-sequence $(n_j)$ such 
that the function
\[
\phi(s) := \lim_j \phi_{n_j}(s) \quad \textrm{exists for all $s \in G$}.
\]
In particular, by dominated convergence, we have
\[
\phi(e) = 0 \qand \sum_s \phi(s) \, \mu(s) = \ell_\mu(\rho). 
\]
It remains to prove that $\phi$ is quasi-$\mu$-harmonic and left Lipschitz. Since we can write 
\[
\rho(gst) - \rho(t) = \rho(gst) - \rho(st) + \rho(st) - \rho(t), \quad \textrm{for all $g,s,t \in G$},
\]
we have $|\phi_n(gs) - \phi_n(s)| \leq \rho(g)$ for all $s$ and $n$, and thus $\phi$ is left Lipschitz. Furthermore, 
\begin{equation}
\label{diag}
\sum_s f_k(gs) \, \mu(s) = f_{k+1}(g) + \sum_s f_k(s) \, \mu(s)
\end{equation}
for all $k$. Hence, 
\[
\sum_s \phi(gs) \, \mu(s) = \lim_j \frac{1}{n_j}  \sum_{k=0}^{n_j-1} \sum_s f_k(gs) \, \mu(s) = 
\lim_j \frac{1}{n_j} \sum_{k=0}^{n_j-1} f_{k+1}(g) + \lim_j \frac{1}{n_j}  \sum_{k=0}^{n_j-1} \sum_s f_{k}(s) \mu(s),
\]
which clearly converges to $\phi(g) + \ell_\rho(\mu)$ for all $g \in G$.
\end{proof}

\begin{proof}[Proof of Lemma \ref{2.2}]
Fix $g \in G$ and define $u(s) = \phi(gs) - \phi(s)$. Since $\phi$ is left Lipschitz and quasi-$\mu$-harmonic, we readily see
that $u$ is a bounded $\mu$-harmonic function, and thus constant since $(G,\mu)$ is Liouville. By evaluating $u$
at $e$ and using $\phi(e) = 0$, we conclude that $\phi(gs) - \phi(s) = \phi(g)$ for all $g,s \in G$ and thus $\phi$ is a
homomorphism. 
\end{proof}

\section{Ergodicity of products}

The main aim of this section is to prove the following theorem.

\begin{theorem}
\label{dbl}
Let $(Z,m)$ denote the Poisson boundary of $(G,\mu)$ and let $(X,\check{\nu})$ be an ergodic $(G,\check{\mu})$-space.
Then the diagonal action $G \acts (Z \times X, m \otimes \check{\nu})$ is ergodic.
\end{theorem}

As an immediate consequence of this theorem, we deduce an important special case of a result by 
Kaimanovich \cite{K03}:

\begin{corollary}
Let $(Z,m)$ and $(\check{Z},\check{m})$ denote the Poisson boundaries of $(G,\mu)$ and $(G,\check{\mu})$ 
respectively. Let $(Y,\eta)$ be an ergodic probability measure-preserving $G$-space. Then the diagonal action
$G \acts (Z \times \check{Z} \times Y, m \otimes \check{m} \otimes \eta)$ is ergodic.
\end{corollary}

\begin{proof}
We shall use Theorem \ref{dbl} twice. First, note that since $\eta$ is preserved by the $G$-action, the product space $(\check{Z} \times Y, \check{m} \times \eta)$ is an ergodic $(G,\check{\mu})$-space by Theorem \ref{dbl}. Hence, by Theorem \ref{dbl}, the diagonal $G$-action on the product of $(Z,m)$ and $(\check{Z} \times Y, \check{m} \times \eta)$ is ergodic.
\end{proof}

We shall prove that if $F \in L^\infty(Z \times X,m \otimes \check{\nu})$ is essentially $G$-invariant and 
$\int_{Z \times X} F \, dm \otimes \check{\nu} = 0$, then $F = 0$ almost everywhere, or equivalently, for every 
$\phi \in L^1(Z,m)$, 
\[
\int_{Z} \phi(z) \, F(z,x) \, dm(z) = 0, \quad \textrm{for $\check{\nu}$-a.e. $x \in X$}.
\] 
Since the linear span of all functions of the form $\frac{ds_*m}{dm}$, where $s$ ranges over $G$, is norm dense in the Banach space $L^1(Z,m)$ by \eqref{corfur}, we 
are left with proving
\[
\int_Z \frac{ds_*m}{dm}(z) \, F(z,x) \, dm(z) = \int_Z F(sz,x) \, dm(z) = \int_Z  F(z,s^{-1}x) \, dm(z) = 0
\]
for $\check{\nu}$-a.e. $x \in X$ and for all $s \in G$. Let $f(x) = \int_Z F(z,x) \, dm(z)$ and note that 
$f \in L^\infty(X,\check{\nu})$ and
\[
\sum_{s} f(sx) \, \check{\mu}(s) 
= \sum_s \Big( \int_Z F(sz, x) \, dm(z) \Big) \, d\mu(s) = f(x),
\]
since $m$ is $\mu$-stationary. We wish to prove that $f$ vanishes $\check{\nu}$-almost everywhere, or, what amounts to the same, that $f$
is $\check{\nu}$-essentially constant. Indeed, if $f$ is essentially constant, then
\[
f(x) = \int_Z F(z,x) \, dm(z) = \int_{Z \times X} F \, dm\otimes \check{\nu} = 0, \quad \textrm{for $\check{\nu}$-a.e $x \in X$},
\]
and thus the following lemma, applied to the ergodic $(G,\check{\mu})$-space $(X,\check{\nu})$, finishes the proof of Theorem \ref{dbl}.

\begin{lemma}
\label{statinv}
Let $(X,\nu)$ be a $(G,\mu)$-space. If $f \in L^\infty(X,\nu)$ satisfies 
\[
\sum_{s} f(sx) \, \mu(s) = f(x), \quad \textrm{for $\nu$-a.e. $x \in X$}, 
\]
then $f$ is $G$-invariant. In particular, if $(X,\nu)$ is ergodic, then $f$ is 
essentially constant.
\end{lemma}

\begin{proof}
Since the support of $\mu$ generates $G$ as a semi-group, it suffices to show that
\[
\sum_{s} \Big( \int_X \big| f(s x) - f(x) \big|^2 \, d\nu(x) \Big) \, \mu^{*k}(s) = 0
\]
for all $k$. Upon expanding the square and using the fact that $(X,\nu)$ is a $(G,\mu)$-space, we note that 
\[
\sum_s \Big( \int_X (f(s x) - f(x))^2 \, d\nu(x) \Big) \, \mu^{*k}(s) 
=
2 \Big( \int_X f^2 \, d\nu - \int_X f(x) \Big( \sum_s f(s x) \, \mu^{*k}(s) \Big) \, d\nu(x) \Big).
\]
By our assumption on $f$, we conclude that these expressions vanish for every $k$.
\end{proof}

\section{Failure of weak-mixing}

We recall that a \emph{factor} of a $G$-space $(X,\nu)$ is a $G$-space $(W,\xi)$ together with a $G$-equivariant Borel map $p : X' \ra W$, where $X' \subset X$ is a $G$-invariant $\nu$-conull set such that 
\[
p_*\nu(A) = \nu(p^{-1}(A)) = \xi(A), \quad \textrm{for all Borel sets $A \subset W$}.
\]
We say that the $G$-spaces $(X,\nu)$ and $(W,\xi)$ are \emph{isomorphic} if $p$ admits a measurable inverse
map $q : W' \ra X$, defined on a $G$-invariant subset $W' \subset W$, such that $q_*\xi = \nu$. \\ 

Let $K$ be a compact and second countable group and suppose that $\tau : G \ra K$ is a homomorphism with dense image. Given a closed subgroup $L < K$, we denote by $m_{K/L}$ the Haar probability measure on $K/L$, and we note that $(K/L,m_{K/L})$ is a $G$-space under the $G$-action $gkL = \tau(g)kL$, for the coset $kL \in K/L$. We say that a $G$-space $(W,\xi)$ is \emph{isometric} if it is isomorphic to a $G$-space of the form $(K/L,m_{K/L})$ for some
compact and second countable group $K$, closed subgroup $L$ of $K$ and a homomorphism $\tau : G \ra K$ with dense image. Since we assume that our compact groups are second countable, the associated $L^2$-spaces (with respect to the Haar probability measures) are separable in the weak topology. \\

The main aim of this section is to prove the following theorem.

\begin{theorem}
\label{thm3}
Let $(X,\nu)$ be an ergodic $(G,\mu)$-space and $(Y,\eta)$ an ergodic probability measure-preserving $G$-space.
Suppose that the diagonal action $G \acts (X \times Y, \nu \otimes \eta)$ is \emph{not} ergodic. Then there exists a
non-trivial factor of $(X,\nu)$ which is an isometric $G$-space.
\end{theorem}

\begin{remark}
In a remark in an earlier pre-print of this paper, we discussed an alternative approach to this theorem, based on some dynamical properties of the WAP-compactification of $G$ established earlier by Furstenberg-Glasner in \cite{FG10}. Recently, Glasner-Weiss expanded on this remark, and gave a full alternative proof of Theorem \ref{thm3}
in \cite{GW}.
\end{remark}

We need some notation. Let 
$B_1$ denote the unit ball in the \emph{real} Hilbert space $L^2_o(Y,\eta)$ of real-valued square-integrable 
functions with zero integrals, identified up to null sets, endowed with the weak topology, which 
makes $B_1$ into a compact and second countable space. The
regular (Koopman) representation $\pi$ on $L^2_o(Y,\eta)$ defined by $\pi(s)u(y) = u(s^{-1}y)$ is unitary, and gives rise to a weakly continuous action on $B_1$ via $su = \pi(s)u$ for $s \in G$ and $u \in B_1$. Theorem \ref{thm3} is an immediate consequence of the following two lemmata. 

\begin{lemma}
\label{3.1}
Suppose that $f \in L^\infty(X \times Y, \nu \otimes \eta)$ is essentially $G$-invariant and has zero $\nu \otimes \eta$-integral, but does not vanish almost everywhere. If $\|f\|_\infty \leq 1$, then $p_f : X \ra  B_1$ given by $p_f(x) = f(x,\cdot)$ is a factor map, where $\nu_f = (p_f)_*\nu$. Furthermore, the Borel function 
\[
\Phi(u,v) =  \langle u, v \rangle_{L^2_o(Y,\eta)}, \quad \textrm{for $u,v \in B_1$}, 
\] 
is \emph{not} $\nu_f \otimes \nu_f$-essentially constant on $B_1 \times B_1$. In particular, $(B_1,\nu_f)$ is a non-trivial ergodic $(G,\mu)$-space.
\end{lemma}

\begin{remark}
To see why $p_f(x) \in B_1$, note that $\|p_f(x)\|_{L^2} \leq \|p_f(x)\|_\infty \leq 1$ and the function
\[
h(x) = \int_Y p_f(x) \, d\nu = \int_Y f(x,y) \, d\eta(y) 
\]
is essentially $G$-invariant, and thus essentially constant by ergodicity of $(X,\nu)$. Since $f$ is assumed to 
have $\nu \otimes \eta$-integral equal to zero, we conclude that $h$ vanishes almost everywhere, and thus 
$p_f(x) \in L^2_o(Y,\eta)$ for $\nu$-a.e. $x$ in $X$. 
\end{remark}

In the next lemma, $\cH$ denotes a real (separable) Hilbert space with inner product $\langle \cdot, \cdot \rangle$
and $B_1$ denotes the unit ball in $\cH$ endowed with the weak topology. We assume that $\pi$ is a unitary representation of $G$ on $\cH$, and consider $B_1$ as a $G$-space under the action $gu = \pi(g)u$. 

\begin{lemma}
\label{3.2}
Suppose that $\xi$ is a $\mu$-stationary Borel probability measure on $B_1$. Then $\xi$ is $G$-invariant. Furthermore, if the Borel map $(u,v) \mapsto \langle u, v \rangle$
is \emph{not} $\xi \otimes \xi$-essentially constant on $B_1 \times B_1$, then $(B_1,\xi)$ admits a 
non-trivial factor which is isometric.
\end{lemma}

\begin{proof}[Proof of Lemma \ref{3.1}]
The assertion that $p_f$ defines a factor map into $(B_1,\nu_f)$ is easy, so what remains to prove is that $\Phi$ is not $\nu_f \otimes \nu_f$-almost everywhere equal to a constant $c$, or, what amounts to the same thing, the function 
$f_2$ on $X \times X$ defined by 
\[
f_2(x,z) := \Phi(p_f(x),p_f(z)) = \langle p_f(x), p_f(z) \rangle_{L^2_o(Y,\eta)} = \int_Y f(x,y) f(z,y) \, d\eta(y)
\]
is not $\nu \otimes \nu$-almost everywhere equal to $c$. We first claim that if  $f_2 = c$ almost everywhere, then
$c = 0$. Indeed, note that
\begin{equation}
\label{defc}
c = \int_X \int_X f_2(x,z) \, d\nu(x) d\nu(z) = \int_Y \Big( \int_X f(x,y) \, d\nu(x)  \Big)^2 \, d\eta(y).
\end{equation}
Since $f$ is $G$-invariant and $(X,\nu)$ is a $(G,\mu)$-space, the function
\[
\Lambda(y) = \int_X f(x,y) \, d\nu(x)
\]
satisfies 
\[
\sum_{s \in G} \Lambda(s y) \, \check{\mu}(s) 
= 
\sum_{s \in G} \Big( \int_X f(x,s^{-1}y) \, d\nu(x) \Big) \mu(s)
=
\sum_{s \in G} \Big( \int_X f(sx,y) \, d\nu(x) \Big) \mu(s) = \Lambda(y).
\]
By Lemma \ref{statinv}, applied to $\Lambda$ and the $(G,\check{\mu})$-space $(Y,\eta)$ - note that every probability measure preserving $G$-space is automatically a $(G,p)$-space for any probability measure $p$ on $G$ - we conclude 
that $\Lambda$ is $G$-invariant and thus essentially constant by ergodicity of $(Y,\eta)$. Since 
\[
\int_Y \Lambda(y) \, d\eta(y) = \int_X \int_Y f(x,y) \, d\nu(x) \, d\eta(y) = 0,
\]
by our assumption on $f$, we see that $\Lambda$ vanishes almost everywhere. From \eqref{defc}, we conclude that $c = 0$. Hence it suffices to prove that $f_2$ does not vanish $\nu \otimes \nu$-almost everywhere. Assume that $f_2(x,z) = 0$ almost everywhere, so that 
\[
\int_{X \times X} f_2(x,z) \psi(x) \psi(z) \, d\nu(x) \, d\nu(z) = \int_Y \Big| \int_X f(x,y) \psi(x) \, d\nu(x) \Big|^2 \, d\eta(y) = 0, 
\]
for all $\psi \in L^\infty(X,\nu)$. Then, 
\[
\int_X f(x,y) \, \psi(x) \, d\nu(x) = 0, \quad \textrm{for $\eta$-a.e. $y$ and for all $\psi \in L^\infty(X,\nu)$},
\]
which readily implies that $f$ vanishes identically, and this contradiction finishes the proof.
\end{proof}

\begin{proof}[Proof of Lemma \ref{3.2}]
We first prove that $\xi$ is $G$-invariant. Let $C(B_1)$ denote the real Banach space of continuous functions on $B_1$, equipped with the uniform norm and let $\cF \subset C(B_1)$ be the sub-algebra of $C(B_1)$ generated 
by the constant functions and all functions on the form
\[
\phi(v) = \langle u_1,v \rangle \cdots \langle u_k, v \rangle, \quad \textrm{for $v \in B_1$},
\]
where $u_1,\ldots,u_k$ range over all finite lists of vectors in $B_1$. Since $\cF$ separates points and does not 
vanish anywhere, we conclude by Stone-Weierstrass Theorem, that $\cF$ is uniformly dense in $C(B_1)$. For a 
fixed list $u_1,\ldots, u_k$ and with $\phi$ as above, we note that
\[
\int_{B_1} \phi(sv) \, d\xi(v) = \int_{B_1} \langle u_1, \pi(s)v \rangle \cdots \langle u_k, \pi(s)v \rangle \, d\xi(v)
= \langle u_1 \otimes \cdots \otimes u_k, \pi^{\otimes k}(s) \sigma_k \rangle_{\cH^{\otimes k}},
\]
where $\sigma_k = \int_{B_1} v \otimes \cdots \otimes v \, d\xi(v)$. Since $\xi$ is $\mu$-stationary and $u_1,\ldots, u_k$
are arbitrary, we conclude that 
\[
\sum_{s \in G} \mu(s) \, \pi^{\otimes k}(s)\sigma_k  = \sigma_k.
\]
Since norm balls in $\cH^{\otimes k}$ are strictly convex, we see that $\sigma_k$ is fixed by all $\pi^{\otimes k}(s)$
when $s$ ranges over $\supp \mu$. Since the support of $\mu$ is assumed to generate $G$, we conclude that 
$\sigma_k$ is fixed by $G$, and thus 
\[
\int_{B_1} \phi(sv) \, d\xi(v) = \int_{B_1} \phi(v) \, d\xi(v), \quad \textrm{for all $s \in G$ and $\phi \in \cF$}.
\] 
Since $\cF$ is uniformly dense in $C(B_1)$, we see that $\xi$ is $G$-invariant. \\

Since $\Phi(u,v) = \langle u, v \rangle$ is a $G$-invariant Borel function on $B_1 \times B_1$ which is not essentially
constant with respect to $\xi \otimes \xi$, there exists by Mackey's Theorem \cite{Ma}, a non-trivial isometric factor of $(B_1,\xi)$.  
\end{proof}

\section{Acknowledgements}

The author is very grateful to the careful referee who has spotted and corrected numerous misprints and inaccuracies in earlier pre-prints of this paper. The author also wishes to thank Uri Bader, Alex Furman, Eli Glasner, Tobias Hartnick, Vadim Kaimanovich, Anders Karlsson, Amos Nevo, Yuval Peres and Benjy Weiss for many enlightening and encouraging discussions on the topics of this paper.

\end{document}